\documentclass[a4paper,12pt,oneside]{article}
\usepackage[top=3cm, bottom=3cm, left=3cm, right=3cm, driver=dvipdfm]{geometry}

\usepackage{amsmath}
\usepackage{amsfonts}
\usepackage{hyperref}
\usepackage{xcolor}
\hypersetup{
   colorlinks,
   linkcolor={blue!50!black},
   citecolor={red!40!black},
   urlcolor={green!60!black}
}
\usepackage{tikz}

\usepackage{amsthm}
\theoremstyle{plain}
\newtheorem{theorem}{\bfseries Theorem}[section]
\newtheorem{lemma}[theorem]{\bfseries Lemma}
\newtheorem{corollary}[theorem]{\bfseries Corollary}
\theoremstyle{definition}
\newtheorem{definition}[theorem]{\bfseries Definition}

\newcommand{\lmu}{\underline{\mu}}
\newcommand{\umu}{\overline{\mu}}
\newcommand{\dd}{\ensuremath{\mathbf{d}}}

\newcommand{\ceil}[1]{\left\lceil #1 \right\rceil}
\newcommand{\mms}{\ensuremath{\mathrm{mms}}}
\newcommand{\red}{\ensuremath{\mathrm{red}}}
\newcommand{\blue}{\ensuremath{\mathrm{blue}}}

\newcommand{\regseq}{\mathbf{k^{n}}}

\begin{document}
\title{The Manickam-Mikl\'os-Singhi Parameter of Graphs and Degree Sequences}

\author{Zolt\'an Kir\'aly \\[-0.8ex]
\small{Department of Computer Science, E\"otv\"os Lor\'and University, Budapest}
\and Neeraja Kulkarni \\[-0.8ex]
\small{Department of Mathematics and Statistics, Carleton College}
\and Ian McMeeking \\[-0.8ex]
\small{Department of Mathematics, University of Minnesota Twin Cities}
\and Joshua Mundinger \\[-0.8ex]
\small{Department of Mathematics and Statistics, Swarthmore College}}
\maketitle
\begin{abstract}
Let $G$ be a simple graph. 
Consider all weightings of the vertices of $G$ with real numbers whose total sum is nonnegative. 
How many edges of $G$ have endpoints with a nonnegative sum?
We consider the minimum number of such edges over all such weightings as a graph parameter. Computing this parameter has been shown to be NP-hard but we give a polynomial algorithm to compute the minimum of this parameter over realizations of a given degree sequence. We also completely determine the minimum and maximum value of this parameter for regular graphs.\footnote{This research was done in 2016, during a Research course at Budapest Semesters in Mathematics.}
\end{abstract}

\section{Introduction}

Suppose there are $n$ real numbers $x_1, \ldots, x_n$ with nonnegative sum.
How many subsums $x_{i_1} + \cdots + x_{i_k}$ of size $k$ are there which are also nonnegative?
The Manickam-Miklós-Singhi (MMS) Conjecture is that if $n$ is at least $4k$, then there are at least $\binom{n-1}{k-1}$ nonnegative subsums.
The conjecture was proven for $n \geq \min\{33k^2, 2k^3\}$ by Alon, Huang, and Sudakov \cite{alon12} and for $n \geq 10^{46}k$ by Pokrovskiy \cite{pokrovskiy15}.

For a function $w: V \to \mathbb R$ and a set $X\subseteq V$, let $w(X)=\sum_{x\in X} w(x)$. 
% Unnecessary, moreover E was not defined.
%For each $e\in E$, let $w(e) = \sum_{v\in e} w(v)$, where each edge $e$ is considered to be a subset of $V$.
For a hypergraph $H=(V,E)$ let $\nu(H), \; \tau(H),\; \nu^*(H)$ and $\tau^*(H)$ denote the matching number, the cover number, the fractional matching number and the fractional cover number of $H$. It is well known that $\nu(H)\le \nu^*(H)=\tau^*(H)\le \tau(H)$. For $E'\subseteq E$ let $H-E'=(V,E\setminus E')$ be the hypergraph we get from $H$ after deleting the edges in $E'$. 
Moreover, for a $k$-uniform hypergraph $H$ on $n$ vertices, let
\[
\mms(H)=\min_{w: V \to \mathbb R\;;\; w(V)\ge 0}
\bigl(|\{e\in E\;;\; w(e)\ge 0\}|\bigr), {\mathrm{ \ and}}
\]
\[
\mu(H)=\min_{E'\subseteq E}\bigl(|E'|\;;\; \nu^*(H-E')=\tau^*(H-E')<\frac{n}{k}\bigr).
\]

The definition of the hypergraph parameter $\mms(H)$ was introduced by
D.\ Mikl\'os \cite{miklos}, inspired by \cite{alon12} and \cite{huang14}.
The following theorem was proved in \cite{alon12} for complete uniform hypergraphs. We repeat their proof with a slight modification for this more general statement.

\begin{theorem}\label{thm_hypergraph}
For any $k$-uniform hypergraph $H$, $\mms(H)=\mu(H)$.
\end{theorem}

\begin{proof}
First take a weighting $w$ for which $w(V)\ge 0$ and $|\{e\in E\;;\; w(e)\ge 0\}|=\mms(H)$. Let $E'= \{e\in E\;;\; w(e)\ge 0\}$, so $\mms(H)=|E'|$. After dividing each weight by $2k\cdot\max_{v\in V}(|w(v)|)$, we may assume that $w(v)<1/k$ for all $v\in V$. There is an $\varepsilon>0$ such that for $w'(v)=w(v)+\varepsilon$ we have $w'(v)\le 1/k$ for every $v\in V$, and also $w'(e)<0$ whenever $w(e)<0$. Clearly $w'(V)>w(V)\ge 0$.
Let $f(v)=1/k-w'(v)$ for each $v\in V$. Now $f$ is a fractional cover of $H-E'$ with size $f(V)=n/k-w'(V)<n/k$, proving $\mu(H)\le \mms(H)$.

Let $E'\subseteq E$ be a subset with $\mu(H)=|E'|$ and $\tau^*(H-E')<\frac{n}{k}$. Let $f$ denote a fractional cover of $H-E'$ with $f(V)=n/k-\delta<n/k$. That is, for each edge $e\in E\setminus E'$ we have $f(e)\ge 1$. For a vertex $v\in V$, define $w(v)=\frac{1}{k}-\frac{\delta}{n}-f(v)$. On one hand $w(V)=n/k-\delta-f(V)=0$, on the other hand,  for any $e\in E$ if $w(e)=1-k\delta/n-f(e)\ge 0$,
then $f(e)\le 1-k\delta/n<1$, so $e\in E'$ as $f$ is a fractional cover for $H-E'$. 
\end{proof}

Let $\delta(H)$ denote the minimum degree in $H$.
It is obvious that $\mms(H)=\mu(H)\le\delta(H)$.
Huang and Sudakov in \cite{huang14} defined that a $k$-uniform hypergraph $H$ \emph{has the MMS property} if $\mms(H)=\delta(H)$. Using this concept the MMS conjecture says that if $n\ge 4k$, then the complete $k$-uniform hypergraph on $n$ vertices has the MMS property.

In this paper, only simple graphs ($2$-uniform hypergraphs) are considered.
Let $G=(V,E)$ be a graph. A subgraph $F$ is called a \emph{perfect $2$-matching} if it is spanning and its every component is either a $K_2$ or an odd cycle. ({\sl Remark: in the literature its weighted version is usually called a perfect $2$-matching when we give weight one to the edges of cycles and weight two for the other edges.})
It is well known that $G$ has a perfect $2$-matching if and only if $\nu^*(G)=n/2$. For $S\subseteq V$, let $\Gamma_G(S)$
denote the set of vertices in $V\setminus S$ having at least one neighbor in $S$. In 1953, Tutte characterized the graphs having a perfect $2$-matching.  

\begin{theorem}[Tutte \cite{Tutte53}]\label{thm_tutte}
A graph $G$ has a perfect $2$-matching if and only if every independent set $S$ of vertices satisfy $|\Gamma_G(S)| \geq |S|$. 
\end{theorem}

Putting Theorems \ref{thm_hypergraph} and \ref{thm_tutte} together gives the following corollary, first observed by a previous research group at Budapest Semesters in Mathematics \cite{miklos}:

\begin{theorem}\label{thm: mu-p2m}
	Let $G=(V,E)$ be a graph. Then the following are equivalent:
    \begin{enumerate}
    	\item $\mms(G)=\mu(G) \leq k$;
        \item There exists a set $S\subseteq V$ and a set $E'$ of $k$ edges such that in the graph $G-E'$, $S$ is independent with fewer than $|S|$ neighbors;
        \item There exists a set $E'$ of $k$ edges such that $G-E'$ has no perfect $2$-matching, i.e., $E'$ blocks (covers) every perfect $2$-matching.
    \end{enumerate}
\end{theorem}

\begin{corollary}
Let $G=(V,E)$ be a graph. Then $\mu(G)=0$ if and only if there exists an independent set $S\subseteq V$ such that  $|\Gamma_G(S)| < |S|$.
\end{corollary}

\begin{corollary}
A graph $G$ has the MMS property $($i.e., $\mu(G)=\delta(G))$ if no fewer edges than $\delta(G)$ can block every perfect $2$-matching.
\end{corollary}

For a graph $G=(V,E)$ we define some notation.
The degree of a vertex $v$ is denoted by $d_G(v)$. Let $S,T\subseteq V$ be two disjoint subsets of the vertices.
Let $i_G(S)$  denote the number of edges having both end-vertices in $S$, and let $d_G(S,T)$ denote the number of edges having one end-vertex in $S$ and the other end-vertex in $T$. Moreover, we use the following unusual notation. Let $E_G(S; V\setminus T)$ denote the set of edges having either both end-vertices in $S$ or one end-vertex in $S$ and the other end-vertex in $(V\setminus T)\setminus S$. For simplicity, this latter set will be denoted by $V-S-T$. Thus $|E_G(S; V\setminus T)|=i_G(S)+d_G(S,V-S-T)$.

\begin{corollary} \label{cor: mu-deleted-edges}
	Let $G$ be a graph. Then
    \begin{equation*}
    	\mu(G) = \min |E_G(S; V\setminus T)|,
    \end{equation*}
    where $S$ and $T$ range over all disjoint subsets of $V$ such that $|S| > |T|$.
\end{corollary}

\begin{proof}
Suppose that $E'$ is a set of $\mu(G)$ edges so that some $S$ is independent in $G-E'$ with fewer than $|S|$ neighbors.
    Let $T$ be the neighborhood of $S$ in $G-E'$;
    then $E_G(S; V\setminus T)\subseteq E'$.
    Conversely, for disjoint $S$ and $T$ with $|S| > |T|$, the subgraph
    $G-E_G(S; V\setminus T)$ has $S$ independent with $|T| < |S|$ neighbors.
\end{proof}

If $S$ and $T$ are disjoint subsets of $V$ with $|S|>|T|$ and $\mu(G) = |E_G(S; V\setminus T)|$, then we say that the pair $(S,T)$ \emph{realizes} $\mu(G)$.

\bigskip

As bipartite graphs have no odd cycles, we also get:

\begin{corollary}
Let $G$ be a bipartite graph. Then $\mu(G)$ is the minimum number of edges that can block every perfect matching.
\end{corollary}

 So $\mu(G)$ is a nice graph parameter. One can ask whether it is computable or not.
 
\begin{theorem}[Dourado et al.\ \cite{dourado15}]
	For a bipartite graph $G$, checking whether $\mu(G) <\delta(G)$ is NP-complete.
\end{theorem}

A recent trend in graph theory is the following: given a graph parameter---perhaps one which is NP-hard to compute---what values does that parameter take over all graphs with the same degree sequence?
Dvo\v{r}\'ak and Mohar in \cite{dvorak13} and later  Bessy and Rautenbach in \cite{bessy17} investigated the possible values of the chromatic number and clique number over a given degree sequence, obtaining nice bounds relating them.
Hence, we will investigate the possible values of $\mu(G)$ for all $G$ realizing a given degree sequence.

Given a degree sequence $\dd$, let $\lmu(\dd)$ denote the minimum value of $\mu$ over all graphs with degree sequence $\dd$.
One of our main results is the following:

\begin{theorem}\label{thm_lower-mu-poly}
	Given a degree sequence $\dd$, there is a polynomial algorithm to determine $\lmu(\dd)$. 
\end{theorem}

% We can also investigate the maximum value of $\mu$ over all % graphs with degree sequence $\dd$, denoted by $\umu(\dd)$.
% We give criteria for when $\umu(\dd)$ is positive, using 
% Kundu's $k$-factor theorem.

We also compute the maximum and minimum values of $\mu$ over all regular degree sequences.

\section{Degree sequences}

Given a graph, its \emph{degree sequence} is the list of degrees of the vertices of the graph.

\begin{definition}
	A \emph{graphical degree sequence} is a sequence of nonnegative integers $(d_1, d_2, \ldots, d_n)$ with even sum such that
    there is a simple graph $G$ on vertex set $V=\{v_1, v_2, \ldots, v_n\}$ in which the degree of of vertex $v_i$ is $d_i$ for $1 \le i \le n$.
\end{definition}

\begin{definition}
For a graphical degree sequence $\dd$, we define
\begin{align*}
	\lmu(\dd) &= \min\{ \mu(G) : G\text{ realizes } \dd\}, \\
    \umu(\dd) &= \max\{ \mu(G) : G\text{ realizes } \dd\}.
\end{align*}
We refer to these as \emph{lower $\mu$} and \emph{upper $\mu$} of a degree sequence, respectively.
\end{definition}

A fundamental tool of realizations of a degree sequence is the \emph{swap}: given an alternating 4-cycle of edges and non-edges in a graph, swapping the edges and non-edges gives a new realization of the same degree sequence. 
Given two graphs $G$ and $G'$ which realize the same degree sequence, there always exists a sequence of swaps which may be applied to $G$ to obtain $G'$ (see, for example, \cite[pp.\ 153-154]{berge73}).
We thus investigate the effect of a swap on $\mu$, making use of Corollary \ref{cor: mu-deleted-edges}.

\begin{lemma}\label{lem: mu-swap}
	Suppose $G$ and $G'$ have the same degree sequence
    and can be obtained from each other via a single swap.
    Then 
    \begin{equation*}
    	\left|\mu(G) - \mu(G')\right| \leq 1.
    \end{equation*}
\end{lemma}
\begin{proof}
	Suppose that $V(G) = V(G') = V$
    and that $G$ is changed to $G'$ by a single swap. 
	Let $S$ and $T$ be disjoint subsets of $V$. 
	We compare $E_G(S; V\setminus T)$ and $E_{G'}(S; V\setminus T)$.
	Since $G'$ has exactly two edges that are not also edges of $G$, we observe
	\begin{equation*}
		|E_{G'}(S; V\setminus T)| \leq |E_G(S; V\setminus T)| + 2
	\end{equation*}
	If equality holds, 
	then both of the swapped edges in $G'$ are in $E_{G'}(S; V\setminus T)$,
	while neither of the swapped edges in $G$ are in $E_G(S; V\setminus T)$.
	Say the swapped edges in $G'$ are $su, s'u'$ where $s,s'\in S$ and $u,u'\notin T$. 
	Then either $su'$ or $ss'$ is one of the swapped edges in $G$,
	and is in $E_G(S; V\setminus T)$.
	Hence equality cannot hold.

	If $(S,T)$ realizes $\mu(G)$, we conclude that 
	\begin{equation*}
		\mu(G') \le |E_{G'}(S; V\setminus T)| \le |E_G(S; V\setminus T)| + 1 = \mu(G) + 1.
	\end{equation*}
	By symmetry, $\mu(G) \leq \mu(G') + 1$ as well.
\end{proof}

As any two realizations of a given degree sequence are related by a sequence of swaps, the possible values of $\mu$ over a degree sequence form an interval.
Thus, complete information is given by the minimum and maximum possible values of $\mu$ over a degree sequence.

\begin{theorem}\label{thm: mu-continuity}
	Suppose $\dd$ is a graphical degree sequence, $k$ is an integer
    and $\lmu(\dd) \leq k \leq \umu(\dd)$.
    Then there is a graph $G$ realizing $\dd$
    such that $\mu(G) = k$.
\end{theorem}
\begin{proof}
	Let $G$ and $G'$ be realizations of $\dd$ with $\mu(G) = \lmu(\dd)$ and $\mu(G') = \umu(\dd)$.
    There is a sequence of graphs $G= G_1, G_2, \ldots, G_m = G'$ such that adjacent graphs are related by a swap.
	By Lemma \ref{lem: mu-swap}, $\mu(G_i)$ and $\mu(G_{i+1})$ differ by at most one.
	Hence $k$ must be equal to $\mu(G_i)$ for some $i$.
\end{proof}

\section{Computation}

We now show that
$\lmu(\dd)$ is computable in polynomial time.
This relies on the characterization of $\mu$ given in Corollary \ref{cor: mu-deleted-edges}.
The fundamental strategy for computing $\lmu(\dd)$ is the following. For any fixed disjoint pair $(S,T)$ of subsets of $V=\{v_1,\ldots,v_n\}$ with $|S|>|T|$, first compute the minimum of $|E_G(S; V\setminus T)|$ over all graphs $G$ on $V$ realizing $\dd$.
Then compute the minimum of these minimum values over all pairs $(S,T)$.

Of course, there are too many pairs $(S,T)$ of disjoint subsets.
The following lemmas help to reduce the number of pairs $(S,T)$ we need to check.

\begin{lemma}\label{lem: colored-vertex-swap}
Suppose that $G$ is a graph on vertex set $V$, and $V$ is colored by red and blue. Let $v_1, v_2$ be vertices in $V$ such that $d_G(v_1) \geq d_G(v_2)$. 
Then there exists a graph $G'$ on $V$ such that $d_G(v) = d_{G'}(v)$ for all $v \in V$,
and such that $v_1$ has at least as many red neighbors in $G'$ as $v_2$ had in $G$.
\end{lemma}

\begin{proof}
Assume that $v_1$ has strictly fewer red neighbors than $v_2$ in $G$ (if this is not the case, then $G'=G$ is good). We can find a red vertex $v_{\red}$ that is a neighbor of $v_2$ but not of $v_1$. 
Since $d_G(v_1) \geq d_G(v_2)$, $v_1$ has strictly more blue neighbors  than $v_2$, so there is a blue vertex $v_{\blue}$ which is a neighbor of $v_1$ but not of $v_2$. 
Swap  edges $v_2v_{\red}$ and $v_1v_{\blue}$ so that $v_1$ has one more and $v_2$ has one less red neighbor. 
To form $G'$, repeat this process until $v_1$ has at least as many red neighbors as $v_2$ had in $G$.
\end{proof}

\begin{lemma} \label{lem: S-T-degrees}
	Let $G$ be a graph with vertex set $V$, and $S$ and $T$ be disjoint subsets of $V$.
    Then there exists another graph $G'$ on $V$ and disjoint subsets $S'$ and  $T'$ of $V$ with the same sizes as $S$ and $T$
    such that $d_G(v) = d_{G'}(v)$ for all $v \in V$,
    the vertices of $S'$ are of lowest degree in $G'$, 
    the vertices of $T'$ are of highest degree in $G'$,
    and
    \begin{equation*}
    	|E_{G'}(S'; V\setminus T')| \leq |E_G(S; V\setminus T)|.
    \end{equation*}
\end{lemma}

\begin{proof}
Consider all graphs $G'$ on $V$ such that $d_G(v) = d_{G'}(v)$ for all $v \in V$, and all disjoint subsets $S'$ and $T'$ of $V$ satisfying $|S'| = |S|,\; |T'| = |T|$, and 
\[|E_{G'}(S'; V\setminus T')| \leq |E_G(S; V\setminus T)|.\]
Choose the $G'$, $S'$, and $T'$ that minimize the difference $\gamma:=\sum_{s\in S'} d_{G'}(s)-\sum_{t\in T'} d_{G'}(t)$.
Define $U' := V-S'-T'$.
It is claimed that $S'$ consists of vertices of lowest degree and $T'$ consists of vertices of highest degree. 
If not, one of three cases must hold.
\vskip 1 em
\noindent \textbf{Case 1.} \enskip
There exists $v_1 \in S'$ and $v_2 \in T'$ such that $d(v_2) < d(v_1)$.
\\
Color $U' \cup S'$ red.
Apply the swaps described in Lemma \ref{lem: colored-vertex-swap} if necessary so that $v_1$ has at least as many red neighbors in the new graph as $v_2$ had in $G'$.
Then exchange $v_1$ and $v_2$ so that $S'$ now contains $v_2$ in place of $v_1$ and $T'$ contains $v_1$ in place of $v_2$.

\noindent \textbf{Case 2.} \enskip
There exists $v_1 \in U'$ and $v_2 \in T'$ such that $d(v_2) < d(v_1)$.
\\
Color $S'$ red.
As in the previous case, apply the swaps of Lemma \ref{lem: colored-vertex-swap} so that $v_1$ has at least as many red neighbors in the new graph as $v_2$ had in $G'$.
Then exchange $v_1$ and $v_2$ so that $U'$ contains $v_2$ in place of $v_1$ and $T'$ contains $v_1$ in place of $v_2$.

\noindent \textbf{Case 3.} \enskip
There exists $v_1 \in S'$ and $v_2 \in U'$ such that $d(v_2) < d(v_1)$.
\\
Color $U'$ red.
Apply the swaps of Lemma \ref{lem: colored-vertex-swap} so that $v_1$ has at least as many red neighbors in the new graph as $v_2$ had in $G'$.
Then exchange $v_1$ and $v_2$ so that $S'$ contains $v_2$ in place of $v_1$ and $U'$ contains $v_1$ in place of $v_2$.

In every case, Lemma $\ref{lem: colored-vertex-swap}$ shows that the quantity $|E_{G'}(S'; V\setminus T')|$ has not increased, while $\gamma$ has decreased, which is a contradiction. Thus, it must be that $S'$ and $T'$ contain the vertices of lowest and highest degree, respectively.
%by 
%\begin{equation*}
%\sum_{s=n-|S|+1}^{n} d_G(v_s) - \sum_{t=1}^{|T|} d_G(v_t),
%\end{equation*}
%it's not important what the bound is, and I don't think it's hard to see that it's bounded below. JM 7.2.17
%noted. proof switched back to starting with extremal choice of G', S', T' instead of inductively altering the graph to achieve extremal value.  Ian 2/18
\end{proof}

\begin{lemma} \label{lem: choose-S-smaller}
	If $G$ is a graph on $n$ vertices, then
    there exist $S,T$ which are disjoint subsets of $V(G)$
    such that $(S,T)$ realizes $\mu(G)$, and $|S| \leq (n+1)/2$ and $|T| = |S| - 1$.
\end{lemma}

\begin{proof}
	Take any pair $(S, T)$ realizing $\mu(G)$.
    If $|S| \le (n+1)/2$, then we are done, as vertices outside of $S \cup T$ may be added to $T$ (if necessary) to attain $|T| = |S| - 1$.
    If $|S| > (n+1)/2$, then $|T|$ is at most $n - (n+1)/2 = (n-1)/2$.
    In this case, deleting $|S| - (|T|+1)$ vertices from $S$ provides the desired $(S, T)$. (In both cases $|E_G(S;V\setminus T)|$ does not increase.)
\end{proof}

The last ingredient in our algorithm for computing $\lmu(\dd)$ is a polynomial-time algorithm computing the minimum cost $b$-factor.
\begin{definition}
Given a graph $G$ on $n$ vertices and a nonnegative integer weight $b(v)$ for each vertex $v$ of $G$,
a subgraph $F$ is called a \emph{$b$-factor} if $d_F(v)=b(v)$ for every vertex $v\in V$. 
\end{definition}

The weighted problem is the following: given also a nonnegative integer cost $c$ for each edge of $G$,
what is the minimum total cost of a $b$-factor?

By the gadget of Tutte \cite{tutte54} it is easy to reduce this problem to finding a minimum cost perfect matching in a graph having $O(n^2)$ vertices. This later problem is solvable in polynomial time by Edmonds \cite{edmonds65}.

\bigskip

\noindent {\bf Theorem \ref{thm_lower-mu-poly}.} {\sl $\lmu(\dd)$ can be computed in polynomial  time.}

\begin{proof}
We give an algorithm to compute $\lmu(\dd)$.

We may assume that $d_1\ge d_2\ge\dots\ge d_n$.
Given $\dd$, check first whether it is graphical (for example using the linear time algorithm of \cite{kiraly11}). 
Let $K$ be the complete graph on $V=\{v_1\ldots,v_n\}$ and let $b(v_i)=d_i$ for all $i$.
A subgraph of $K$ is a $b$-factor if and only if it realizes $\dd$.

For all $k=1,\dots,\lfloor(n+1)/2\rfloor$, execute the following process:

Let $S=\{v_{n-k+1},\ldots,v_n\}$ and $T=\{v_1,\ldots,v_{k-1}\}$. Define a cost $c(uv)$ of each edge $uv$ of $K$: $c(uv)=0$ unless $u\in S$ and $v\not\in T$, in which case $c(uv)=1$.  
For a given $b$-factor, $G$, the cost of $G$ is exactly $|E_G(S; V\setminus T)|$.
Then calculate a minimum-cost $b$-factor. Call its cost OPT$(k)$.

Finally output $\min_{1\le k\le \lfloor(n+1)/2\rfloor}($OPT$(k))$.  By Lemmas \ref{lem: S-T-degrees} and \ref{lem: choose-S-smaller}, this is exactly $\lmu(\dd)$.
\end{proof}

\section{Existence of perfect 2-matchings}

What is surprising about $\lmu$ is that minimizing over exponentially many realizations of a given degree sequence is possible.
In contrast, it is unknown what the computational complexity of $\umu$ is.

Nonetheless, the relationship between perfect 2-matchings and perfect matchings lets us make some headway in checking whether $\umu$ is positive,
that is, if there exists a realization of a degree sequence with a perfect 2-matching.
Since even cycles have a perfect matching, the only obstruction to having a perfect matching in a graph with positive $\mu$ can be the existence of odd cycles in every perfect $2$-matching.
These may be addressed by the following lemma:
\begin{lemma}\label{lem: odd-cycle-swap}
	Let $G$ be a graph with a perfect 2-matching.
    Then there is another graph $G'$ with the same degree sequence as $G$ and a perfect 2-matching with at most one odd cycle.
\end{lemma}
\begin{proof}
	Let $F$ be a perfect 2-matching in $G$ with the minimum number of odd cycles.
	If $F$ has more than one odd cycle,
	we show how to construct another realization of the degree sequence of $G$
	with a perfect 2-matching with fewer odd cycles. 
    	Say $C_1$ and $C_2$ are distinct odd cycles in $F$.
    Our strategy is to either show  that $G[V(C_1)\cup V(C_2)]$  has a perfect matching or make a single swap in $G$ to achieve the same conclusion.
 
	Let $u_1v_1 \in E(C_1)$ and $u_2v_2 \in E(C_2)$.
	If both $u_1u_2$ and $v_1v_2$ are also edges in $G$,
	we could replace $u_1v_1$ and $u_2v_2$ in $F$ with $u_1u_2$ and $v_1v_2$ getting an even cycle that has a perfect matching.
	But if neither $u_1u_2$ and $v_1v_2$ are edges in $G$,
	then we could swap $u_1v_1$ and $u_2v_2$ to $u_1u_2$ and $v_1v_2$ to form $G'$, now $G'[V(C_1)\cup V(C_2)]$ has a perfect matching.
	
	Thus, we suppose towards contradiction that for all $u_1v_1\in E(C_1)$ and $u_2v_2 \in E(C_2)$, exactly one of $u_1u_2$ and $v_1v_2$ is an edge in $G$.
	Similarly, we may suppose exactly one of $u_1v_2$ and $u_2v_1$ is an edge in $G$.

	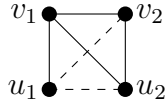
\begin{figure}[h]
		\centering
		\begin{tikzpicture}
			\draw (0,0) node[left](u1) {$u_1$} -- (0,1) node[left](v1){$v_1$} 
			-- (1,0) node[right](u2) {$u_2$} -- (1,1) node[right](v2) {$v_2$};
			\draw (0,1) -- (1,1);
			\draw[dashed] (0,0) -- (1,1);
			\draw[dashed] (0,0) -- (1,0);
			\fill (0,0) circle[radius = 0.1];
			\fill (0,1) circle[radius = 0.1];
			\fill (1,0) circle[radius = 0.1];
			\fill (1,1) circle[radius = 0.1];
		\end{tikzpicture}
		\caption{The situation for all $u_1v_1 \in E(C_1)$ and $u_2v_2 \in E(C_2)$, up to symmetry.}
	\end{figure}
	
	Suppose without loss of generality that $uv$ is an edge in $C_1$ such that both $u$ and $v$ are connected to $w_1\in V(C_2)$. Let the vertices of $C_2$ be (in order) $w_1,\dots,w_\ell$.
	Then both of $w_1$'s neighbors in $C_2$ (i.e., $w_2$ and $w_\ell$) must have no edges to any of $u$ and $v$ by our assumption.
	Similarly, $w_3$ and $w_{\ell-1}$ must be connected to both $u$ and $v$, and so on.  
	This allows us to 2-color $C_2$ according to whether a vertex has neither or both vertices $u$ and $v$ as neighbors, a contradiction since $C_2$ is odd.
	So either $F$ does not have the minimal number of odd cycles,
	or we may swap to a new graph with a perfect 2-matching with fewer odd cycles, as desired.
\end{proof}

For $n$ even, this reduces the existence of a perfect 2-matching to the existence of a perfect matching:
\begin{theorem}
	Let $\dd$ be a graphical degree sequence of even length.
    Then $\dd$ has a realization with a perfect 2-matching if and only if $\dd$ has a realization with a perfect matching.
\end{theorem}
\begin{proof}
	One direction is clear, since a perfect matching is also a perfect 2-matching.
    Conversely, if $\dd$ has a realization with a perfect 2-matching,
    Lemma \ref{lem: odd-cycle-swap} shows that there exists a realization with a perfect 2-matching with at most one odd cycle.
    Since $\dd$ is of even length,
    the number of odd cycles in a perfect 2-matching must be even,
    so our perfect 2-matching in this realization has no odd cycles and hence it is a perfect matching.
\end{proof}

The $k$-factor theorem of Kundu \cite{kundu73} gives (as a special case) conditions for a degree sequence to have a realization with a perfect matching.
In particular, deciding whether a degree sequence has a realization with a perfect matching can be decided in polynomial time.

\begin{theorem}[Kundu \cite{kundu73}] \label{thm: kundu}
	Let $\dd$ be a graphical degree sequence. A realization $G$ of $\dd$ with a perfect matching exists if and only if the sequence $\dd-{\mathbf{1^n}}=(d_1-1,\dots,d_n-1)$ is also graphical. 
\end{theorem}

\begin{corollary}
Let $\dd$ be  a graphical degree sequence with even length. $\umu(\dd)>0$ if and only if $\dd-{\mathbf{1^n}}$ is also graphical. 
\end{corollary}

% The following theorem was not proved.
% I vote for leaving it out (Zoltan).
%
% Kundu's theorem also allows for some partial results when $n$ is odd.
% \begin{theorem}
% 	Suppose that $\dd = (d_1, d_2, \ldots, d_n)$ is a graphical degree sequence with $d_1 \geq d_2 \geq \cdots \geq d_n$.
% 	If $\dd-1^{n-2}21$ is graphical,
%     then $\umu(\dd) > 0$. (Here $\dd-1^{n-2}21$ is a shorthand for $(d_1-1,\dots,d_{n-2}-1, d_{n-1}-2, d_{n}-1)$.
% \end{theorem}

% Although the above condition is sufficient to guarantee $\umu(\dd) > 0$, it is not necessary, as can be seen by the degree sequence of the union of a triangle and a single edge.

% \begin{conjecture}
%  	Suppose that $\dd = (d_1, d_2, \ldots, d_n)$ is a graphical degree sequence with $d_1 \geq d_2 \geq \cdots \geq d_n$ and $n$ odd.
% Then $\umu(\dd) > 0$ iff  $\dd - 1^22^11^{n-3}$ is  graphical.
% \end{conjecture}

\section{Regular degree sequences}

In this section we completely determine the values $\lmu(\dd)$ and $\umu(\dd)$ for all regular degree sequences.

\subsection{The minimal value of $\mu$}

Before computing exact values of $\lmu(\dd)$, we find bounds on the value of $\mu$.  Since deleting all edges incident to a given vertex in $G$ destroys all perfect 2-matchings, we immediately have an upper bound:

\begin{lemma} \label{lem: mu-less-delta}
	For any graph $G$, 
	$\mu(G) \leq \delta(G)$.
\end{lemma}

For regular graphs, the following lower bound also exists.

\begin{lemma} \label{lem: reg-mu-lower-bound}
	For any $k$-regular graph $G$,
    $\mu(G) \geq k/2$.
\end{lemma}

\begin{proof}
	Let $(S,T)$ realize $\mu(G)$.
    After deleting the $\mu$ edges of $E_G(S; V\setminus T)$, the number of edges leaving $S$ is at least $k|S| - 2\mu$.
    However, the number of edges leaving $T$ is at most $k|T| \le k(|S| - 1)$.
    Thus, $k|S| - 2\mu(G) \leq k|S| - k$, so $\mu(G) \geq k/2$.
\end{proof}

We will show the upper bound and the lower bound given by Lemmas $\ref{lem: mu-less-delta}$ and $\ref{lem: reg-mu-lower-bound}$ are tight for large enough $n$ (Theorem $\ref{thm: reg-mu-odd}$). However, if $(S,T)$ realizes $\mu(G)$ and we know the sizes of the sets $S$ and $T$, we can bound $\mu(G)$ more narrowly. The following lemma (together with Lemma \ref{lem: choose-S-smaller}) gives conditional bounds on the sizes of $S$ and $T$.

\begin{lemma} \label{lem: S-is-big}
	Suppose that $G$ has minimum degree $\delta(G) \geq 2$.
    Suppose that $\mu(G) \leq \delta(G) - 1$, and $S$ and $T$ are disjoint sets of vertices of $G$ such that $(S, T)$ realizes $\mu(G)$.
    Then $|S| \geq \delta(G)-1$,
    and if $|S| = \delta(G) - 1$ then $\mu(G) = \delta(G) - 1$.
\end{lemma}

\begin{proof}
	Let us write $s = |S|$ and $\delta = \delta(G)$.
    After deleting the $\mu(G)$ edges of $E_G(S; V\setminus T)$, a vertex in $S$ can have at most $s-1$ neighbors.
    Hence, each vertex of $S$ is incident to at least $\delta - s + 1$ edges to be deleted in $G$.
    Thus, 
    \begin{equation} \label{eq: small-s-bound}
    	\delta-1 \geq \mu(G) \geq \frac{s}{2}(\delta - s + 1).
    \end{equation}
    This reduces to the quadratic $s^2 - (\delta+1)s + 2(\delta-1) \geq 0$,
    which is zero at $s = 2$ or $s = \delta-1$.
    If $s=1$, then $\delta-s+1=\delta$ edges were deleted contradicting $\mu(G) < \delta$. If $s=2$ and $s<\delta-1$, then there are $2(\delta-1)>\delta$ deleted edges, a contradiction again.  
     Hence $s\geq \delta - 1$, as desired.

	If $s = \delta - 1$, then equality holds in \eqref{eq: small-s-bound}, showing that $\mu(G) = \delta - 1$.
\end{proof}

These results allow us to determine the value of lower $\mu$ when $\delta(G)$ is large relative to $n$.
\begin{lemma} \label{lem: 2k-3}
	If $G$ is a graph on $n$ vertices
   	and $\mu(G) < \delta(G)$,
    then $n \geq 2\delta(G) - 3$.
    If in addition $n = 2\delta(G) - 3$,
    then $\mu(G) = \delta(G) - 1$.
\end{lemma}
\begin{proof}
By Lemma \ref{lem: choose-S-smaller},
    we may choose $(S, T)$ realizing $\mu(G)$ such that $|S|$ is at most $(n+1)/2$.
By Lemma \ref{lem: S-is-big}, since $(S, T)$ realizes $\mu(G)$, $|S| \geq \delta(G) - 1$.
So $n+1\ge 2(\delta(G) - 1)$.
If in addition $n = 2\delta(G) - 3$, then $|S| = \delta(G) - 1$, so	Lemma \ref{lem: S-is-big} shows $\mu(G) = \delta(G) - 1$.
\end{proof}

Note that Lemma $\ref{lem: 2k-3}$ applies to all graphs, not only regular ones.
In addition, it proves that the complete graph $K_n$ satisfies $\mu(K_n) = n-1$ for $n \geq 6$, verifying the Manickam-Miklós-Singhi conjecture for graphs.

We use the notation $\regseq$ for the degree sequence $d_1=\dots=d_n=k$. Observe that $\regseq$ is graphical if and only if $k<n$ and $kn$ is even.

\begin{theorem}\label{thm: reg-mu-odd}
	For $n > k$ with $n$ odd and $k$ even,
\begin{equation*}
    	\lmu(\regseq) = \begin{cases}
        	k 	&  {\mathrm{ if \ }}\; n < 2k - 3 \\
            k-1 &  {\mathrm{ if \ }}\; n = 2k - 3 \\
            k/2	&  {\mathrm{ if \ }}\; n \geq 2k-1	.
        \end{cases}
    \end{equation*}
\end{theorem}

\begin{proof}
	Lemma \ref{lem: 2k-3} establishes the case of $n < 2k-3$, and shows that to prove the case of $n=2k-3$,
    it suffices to construct a realization $G$ of the degree sequence $\mathbf{k^{2k-3}}$ with $\mu(G) < k$.
    Start with the complete bipartite graph $K_{k-1, k-2}$. Let its color classes of size $k-1$ and $k-2$ be denoted by $S$ and $T$, respectively. Form $G$ by adding a perfect matching to $T$ and a cycle of size $k-1$ to $S$. The result is a $k$-regular graph. If the cycle in $S$ is deleted, then $S$ is independent with $|T|< |S|$ neighbors. Since the cycle in $S$ has $k-1$ edges, $\mu(G) \leq k-1$, as desired.
    
    Finally, suppose $n \ge 2k-1$.
    By Lemma \ref{lem: reg-mu-lower-bound}, it suffices to find a graph $G$ realizing the degree sequence $\regseq$ with $\mu(G) \le k/2$.
    Begin with a $k$-regular bipartite graph with parts of size $(n+1)/2$,
    which exists since $k \leq (n+1)/2$.
	$G$ is the graph formed by deleting a vertex and adding a perfect matching to the neighborhood of that vertex.
    If one deletes the perfect matching added, then what is left is a bipartite graph with parts of size $(n+1)/2$ and $(n-1)/2$, which cannot have a perfect 2-matching.
	Thus $\mu(G) \leq k/2$, as desired.
\end{proof}

\begin{lemma} \label{lem: n_even_small}
	Let $n$ be even.
    If $G$ is a $k$-regular graph on $n$ vertices
    and $\mu(G) < k$,
    then $n \geq 3k - 2$.
    If in addition $n \leq 3k-1$,
    then $\mu(G) = k - 1$.
\end{lemma}
\begin{proof}
   By Lemma \ref{lem: choose-S-smaller}, there exists $(S, T)$ which realizes $\mu = \mu(G)$ with $|T| = |S|-1$.
    Then $S \cup T$ has an odd number of vertices and so is not all of $V$.
    Let $U = V - S - T$, and let $u = |U|,\; s=|S|,\; t=|T|=s-1$.
Remember that $\mu = |E_G(S; V\setminus T)|=i_G(S)+d_G(S,U)$.
    By counting the number of half-edges incident to vertices of $S$, we obtain
    \begin{equation} \label{eq: n-even-regular-S}
    	ks = 2(\mu - d_G(S,U)) + d_G(S,U) + d_G(S,T) = 2\mu - d_G(S,U) + d_G(S,T).
    \end{equation}
    Exactly $ku - 2i_G(U) - d_G(S,U)$ edges go from $U$ to $T$,
    so by counting the number of half-edges incident to $T$, we obtain
    \begin{equation} \label{eq: n-even-regular-T}
    	kt=k(s-1) \geq (ku - 2i_G(U) - d_G(S,U)) + d_G(S,T).
    \end{equation}
    Since $i_G(U) \leq u(u-1)/2$, 
    subtracting \eqref{eq: n-even-regular-S} from \eqref{eq: n-even-regular-T}, and simplifying gives 
   \[
    	2\mu \geq k(u+1) - 2i_G(U) \geq 2k + (u-1)(k-u).
  \] 
	As $\mu < k$, we see either $u < 1$ or $u > k$.
    But $U$ is non-empty, so $u \geq k + 1$.
    Thus, using Lemma \ref{lem: S-is-big} we get
    \begin{equation*}
		n = u + s + t \geq (k+1) + (k-1) + (k-2) = 3k - 2.
    \end{equation*}
    Since $s+t=2s-1=n-u$, we also have
    \begin{equation} \label{eq: n_even_small_S}
    	s \leq \frac{n - u + 1}{2} \leq \frac{n-k}{2}.
    \end{equation}
    If $n \leq 3k - 1$,
    then \eqref{eq: n_even_small_S} shows $|S| \leq k- 1$,
    so by Lemma \ref{lem: S-is-big}, $\mu(G) = \delta(G) - 1=k-1$.
\end{proof}

\begin{theorem}\label{thm: reg-mu-even}
	For $n > k$ with $n$ even,
    \begin{equation*}
    	\lmu(\regseq) = \begin{cases}
        	k 	&  {\mathrm{ if \ }}\; n < 3k - 2 \\
            k-1	&  {\mathrm{ if \ }}\; n = 3k-2\text{ or }3k-1 \\
            \ceil{k/2}	&  {\mathrm{ if \ }}\; n \geq 3k.
        \end{cases}
    \end{equation*}
\end{theorem}
\begin{proof}
	The case of $n<3k-2$ is immediate from Lemma \ref{lem: n_even_small}.
    To prove the case of $n=3k-1$ or $3k-2$, all that remains is to demonstrate a realization, $G$, of the degree sequence $\regseq$ with $\mu(G)<k$:
    
    If $n=3k-2$, then $k$ must be even as $n$ is even.
    Hence, by Theorem \ref{thm: reg-mu-odd}, there is a $k$ regular graph on $2k-3$ vertices with a $\mu$ value of $k-1$.
    Let $G$ be the disjoint union of this graph with $K_{k+1}$.
    
    Suppose $n=3k-1$.
    Begin with two components.
    One component is the complete graph $K_{k+2}$.
    For the other component, start with the complete bipartite graph $K_{k-1,k-2}$.
    The degree of $k-2$ vertices of one part is $k-1$, and since $k-2$ is odd, it is possible to add a perfect matching to all but one of its vertices, $v_1$.
    The degree of the $k-1$ vertices on the other side is $k-2$, so add a cycle, $C$, of length $k-1$.
	Choose any vertex from the $K_{k+2}$ component, call it $v_2$.
    Let $a_1$ and $a_2$ be any neighbors of $v_2$ in $K_{k+2}$.
    Then delete a perfect matching from the remaining $k-1$ vertices of $K_{k+2}$, so that those $k-1$ vertices now have degree $k$.
    Delete the edges $a_1v_2$ and $a_2v_2$ and add the edge $v_1v_2$. 
    
    Now the graph is $k$-regular on $3k-1$ vertices.
    If the $k-1$ edges of cycle $C$ are deleted, then there is an independent set of size $k-1$ with $k-2$ neighbors.
    Hence $\mu(G)<k$, as desired.
    
    Finally, suppose $n \ge 3k$.
    By Lemma \ref{lem: reg-mu-lower-bound}, it suffices to find a graph $G$ realizing the degree sequence $\regseq$ with $\mu(G) \le \ceil{k/2}$.
    If $k$ is even, let $G'$ be a graph on $n - (k+1)$ vertices with $\mu(G') = k/2$,
    and let $G$ be the disjoint union of $G'$ with a complete graph on $k+1$ vertices.\\
    If $k$ is odd and $n \ge 3k+1$ is even, construct $G$ as follows.
    Let $G'$ be a $k$-regular bipartite graph with parts of size $(n - k - 1)/2$,
    which exists since $n \geq 3k + 1$ implies $(n-k-1)/2 \geq k$.
    Let $H$ be a graph on $k+2$ vertices with degree sequence $(k, k, \ldots, k, k -1)$.
    If $v$ is any vertex of $G'$,
    then $(G' - v) \cup H$ will have $k+1$ vertices of degree $k-1$, while all other vertices have degree $k$.
    Construct $G$ by adding a perfect matching to the $k+1$ vertices of degree $k-1$ in $(G' - v) \cup H$.
    After deleting the $(k+1)/2$ edges in this perfect matching from $G$,
    $G' - v$ is a connected component with $\mu = 0$,
    since it is a bipartite graph with different size color classes.
    Hence $\mu(G) \leq (k+1)/2$.
\end{proof}

\subsection{The maximal value of $\mu$}

This section is devoted to proving that $\umu(\regseq)=k$
whenever $\regseq$ is graphical
(i.e., if $k<n$ and not both $k$ and $n$ are odd), 
except for some sporadic small cases.
We start with the easier case when $n$ is even.

\begin{lemma}\label{lem:umu_n_even}
Suppose $n$ is even and $1\le k<n$. Then $\umu(\regseq)=k$.
\end{lemma}
\begin{proof}
	It is well known that the edge-set of the complete graph on $n$ vertices decomposes into perfect matchings, i.e., $E(K_n)=M_1\cup M_2\cup\dots\cup M_{n-1}$. Let $G=M_1\cup\dots\cup M_k$. Clearly $G$ is $k$-regular and less than $k$ edges cannot block every perfect matching.
\end{proof}

When $n$ is odd, we compute $\mu$ of a certain family of graphs with high symmetry.

\begin{definition}
% For even $k$, the $k$-regular \emph{circulant} on $n$ vertices, denoted $C(k,n)$, is the graph with vertex set residue classes of integers modulo $n$, denoted $\mathbb{Z}/n \mathbb{Z}$, with vertices $i$ and $j$ adjacent if and only if $i-j$ is one of the residue classes $-k/2, -k/2 + 1, \ldots, k/2$.
For even $k$, the $k$-regular \emph{circulant} on $n$ vertices, denoted by $C(k,n)$, is the graph with vertex set $\{1,2,\dots,n\}$, where $i$ and $j$ are adjacent if $|i-j|\le k/2$. For $i<j$, here $|i-j|$ denotes $\min(j-i,\; i+n-j)$.
\end{definition}

% Let us introduce a notation to refer to residue classes. If $i$ is a residue class modulo $n$, then $|i|$ will denote the unique representative of that class which lies in the set $\{0, 1, \ldots, n-1\}$.

\begin{theorem} \label{thm: odd-circulant-mu}
Let $n\ge 9$ be odd, and $4\le k < n$ be even. Then $\mu(C(k,n)) = k$.
\end{theorem}

\begin{proof}
	Let $G = C(k,n)$ and assume towards contradiction that $\mu(G)<k$.
    By Lemma \ref{lem: choose-S-smaller}, there exists $(S, T)$ which realizes $\mu(G)$ with $s=|S|\le (n+1)/2$ and $t=|T|=s-1.$
%    By Lemma \ref{lem: S-is-big}, $s \geq k-1$. 
%     and if $s = k-1$, then $\mu = k-1$. We first show that in fact $s \geq k$.
%   \\  
%     If $s = k-1$,
%     which means $|T| = k-2$, then there are $k(k-1)$ edges leaving the vertices in $S$. Of these, at most $2\mu = 2(k-1)$ may be deleted, which means there are at least $(k-1)(k-2)$ undeleted edges going from $S$ to $T$. Thus every vertex in $T$ is adjacent to all $k-1$ vertices in $S$.
%     But any two distinct vertices of the circulant
%     have at most $k-1$ common neighbors,
%     with exactly $k-1$ common neighbors only if the circulant is the complete graph.
%     And our graph cannot be the complete graph, as the complete graph satisfies $\mu = k$ for $n \geq 6$ by Lemma \ref{lem: 2k-3}. 
% 	Hence, we may assume $s \ge k$.
%     In particular, as $\mu(G) < k$, we have $n \geq 2k-1$.
%     \\ 

   The edges to be deleted to realize $\mu(G)$ are either those internal to $S$ or those that go from $S$ to $\Gamma (S) - T$. The number of these latter edges is at least $|\Gamma_G(S)| - |T|$. Thus,
	\begin{equation} \label{mu_bd}
	k-1 \geq \mu(G) \geq i_G(S) + |\Gamma_G(S)| - |T| = i_G(S) + |\Gamma_G(S)| - s + 1.
	\end{equation}
	Suppose the vertices of $S$ are $v_1=v_{s+1}< v_2< \ldots< v_s$.
	Let $n_i = v_{i+1} - v_i$ for $i=1,\dots,s-1$, and $n_s=v_1+n-v_s$. Thus $\sum n_i=n$.
    
    Between $v_i$ and $v_{i+1}$, there are precisely $\min\{k, n_i - 1\}$ neighbors of either $v_i$ or $v_{i+1}$.
    Hence, 
    \begin{equation} \label{eq: circulant-gamma}
    	|\Gamma_G(S)| = \sum_{i=1}^s \min\{ k, n_i - 1\}.
    \end{equation}
    In addition, if $n_i \leq k/2$, then the vertices $v_i$ and $v_{i+1}$ are adjacent, so
    \begin{equation} \label{eq: circulant-i}
    	i_G(S) \geq |\{i: n_i \leq k/2\}|.
    \end{equation}
    Although \eqref{eq: circulant-i} is a crude estimate,
    it is enough along with \eqref{eq: circulant-gamma}
    to bound $\mu$ for all but a few cases.
	To this end, define $f: \mathbb{N} \to \mathbb{N}$ by
    \begin{equation*}
    	f(n_i) = \begin{cases}
        			\min\{k, n_i - 1\} + 1 & {\mathrm{ if \ }}\; n_i \leq k/2 \\
                    \min\{k, n_i - 1\}	& {\mathrm{ if \ }}\;  n_i \ge  k/2+1.
               \end{cases}
    \end{equation*}
	Now \eqref{mu_bd}, \eqref{eq: circulant-gamma}, and \eqref{eq: circulant-i} yield
	\begin{equation} \label{n_i_contribution}
		k-1 \ge \mu(G) \ge \sum_{i=1}^s f(n_i) -s+1.
	\end{equation}
	Observe that $f(n_i) \ge 1$ for all $i$, with equality if and only if $n_i = 1$.
	We claim that $n_i \leq k$  for all $i$.
    For if $n_i > k$ for some $i$, then $f(n_i) = k$,
    and so since $f(n_j) \geq 1$ for all other $j$,
    \eqref{n_i_contribution} implies that 
    \begin{equation*}
    	k-1 \ge k + (s-1) - s + 1 = k,
    \end{equation*}
    a contradiction.
    So $n_i \leq k$ for all $i$, implying that $|\Gamma_G(S)| = n-s$.
   
   We similarly claim that $n_i \leq k/2$ for all but at most one $i$, yielding $i_G(S)\ge s-1$.
   If instead $n_i \ge k/2+1$ for at least two different values of $i$,
   then again \eqref{n_i_contribution} implies
   \begin{equation} \label{eq: circulant-half-k}
   	k-1 \ge \sum_{i=1}^s f(n_i) - s + 1 \ge 2(k/2+1) + (s-2) - s + 1 = k+1,
   \end{equation}
   a contradiction.
% 	Thus, equality holds in \eqref{eq: circulant-half-k} and so also in \eqref{mu_bd} and \eqref{eq: circulant-i}.
% 	Hence $\mu(G) = k-1$, there are \emph{exactly} two $i$ values for which $n_i = k/2$, and all other $n_i$'s have $f(n_i)=1$, i.e. $n_i = 1$.  
%     If $s > 4$, then by the Pigeonhole Principle, $1 = n_i = n_{i+1}$ for some $i$, that is, $|v_{i+2} - v_i| = 2$.
%     But since $k \geq 4$, $v_i$ and $v_{i+2}$ are adjacent,
%     so equality is impossible in \eqref{eq: circulant-i}.
%     If $s =4$,
%     then $n = \sum_{i=1}^s n_i = k + 2$,
%     contradicting that $n$ is odd.

	Since $i_G(S) \geq s-1$,  \eqref{mu_bd} yields
	\begin{equation*}
	k-1 \geq i_G(S) + |\Gamma_G(S)| - s + 1 \geq (s-1) + (n-s) - s + 1 = n - s.
	\end{equation*}
    Equivalently, $n \leq k - 1 + s \leq 2s$ by Lemma \ref{lem: S-is-big}, however, as $n$ is odd, necessarily $n\le 2s-1$. As $s\le (n+1)/2$, we have $n=2s-1$.

If $k=4$, then $k-1\ge \mu(G)\ge i_G(s)\ge s-1$, yielding $s\le 4$ and thus $n\le 7$, contradicting our assumption that $n\ge 9$. So we may suppose $k\ge 6$.

Now it is time to take edges inside $S$ of form $v_iv_{i+2}$ into account. As $\sum_{i=1}^s (n_i+n_{i+1})=2n=4s-2$, we have at least two different indices $i\ne j$ such that $n_i+n_{i+1}\le 3$ and  $n_j+n_{j+1}\le 3$. Consequently $v_iv_{i+2}$ and $v_jv_{j+2}$ are also edges of $G$, so $i_G(S)\ge s+1$. Now Lemma \ref{lem: S-is-big} and \eqref{mu_bd} yield
\begin{equation*}
s\ge k-1\ge (s+1)+(n-s)-s+1=n-s+2=s+1,
\end{equation*}
a contradiction again.
We have now eliminated all cases; $\mu(G)<k$ is impossible.
\end{proof}

\begin{lemma}\label{lem: 4-7-upper}
If $G$ is a $4$-regular graph on $7$ vertices, then $\mu(G)\le 3$.
\end{lemma}

\begin{proof}
The complement of $G$ is $2$-regular, so it is either a seven-cycle, or the union of a triangle and a $4$-cycle.
In both cases it is easy to find a set $S$ with $|S|=4$, connected by at least $3$ edges of the complement, so $i_G(S)\le 3$. 
\end{proof}

\begin{lemma} \label{lem: C(4,7)}
    $\mu(C(4,7)) = 3$.
\end{lemma}

\begin{proof}
 By Lemma \ref{lem: choose-S-smaller}, there exists $(S, T)$ which realizes $\mu(C(4,7))$ with $s=|S|\le (n+1)/2=4$ and $t=|T|=s-1.$  By Lemma \ref{lem: S-is-big}, $s\ge k-1=3$, and if $s=3$, then $\mu(C(4,7))=3$.
It is easy to see that if $s=4$, then $i_{C(4,7)}(S)\ge 3$.
\end{proof}

\begin{theorem} \label{thm: regular-upper-mu}
	For all $n > k$ such that $\regseq$ is graphical,
    $\umu(\regseq) = k$ unless $k=2$ and $n$ is odd,
    or $k = 4$ and $n = 5$ or $n=7$.
    In these exceptional cases $\umu(\regseq) = k-1$.
\end{theorem}
\begin{proof}

If $n$ is even, then use Lemma \ref{lem:umu_n_even}. From now on we assume that $n$ is odd.
If $G$ is a 2-regular graph on an odd number of vertices, then $G$ must contain an odd cycle as a component.
    Any odd cycle has $\mu = 1$, for deleting one edge leaves an odd path,
    which contains no cycles and does not have a perfect matching.
    Since $G$ is a union of cycles, $\mu(G) = 1$.
    Hence for $n$ odd, $\umu(\mathbf{2^n}) = 1$.
    
    When $k =4$ and $n = 5$, then $K_5$ is the unique realization of $\regseq$. 
    If we delete the edges of any triangle from $K_5$,
    the vertices contained in that triangle become independent with two neighbors.
    This means that $\mu(K_5) \leq 3$. But we know from Lemma \ref{lem: 2k-3} that $\mu(K_5) \geq 3$, 
    thus $\mu(K_5) =3$ and $\umu(\mathbf{4^5}) = 3$.
    
    When $k = 4$ and $n = 7$, Lemmas \ref{lem: C(4,7)} and \ref{lem: 4-7-upper} together show that $\umu(\mathbf{7^4}) = 3$.
    
    Finally, Lemma \ref{lem: 2k-3} and Theorem \ref{thm: odd-circulant-mu} show that $\mu(\regseq) = k$ in all other cases.
\end{proof}

\subsection*{Acknowledgments}

The authors are grateful to Budapest Semesters in Mathematics (BSM) for organizing research courses and also to Dezs\H o Mikl\'os for suggesting this topic.
Thanks also to the spring 2016 BSM research group which established and shared some preliminary results suggesting the direction for the current work.

\bibliographystyle{plain}
\bibliography{references}

\end{document}